\documentclass[a4paper,10pt]{amsart}
\usepackage{latexsym}
\usepackage{amssymb}
\usepackage{amsfonts}
\usepackage{amsmath}
\usepackage{indentfirst}
\usepackage{graphicx}
\usepackage{epsfig}

\usepackage{epic, eepic}

\newtheorem{theorem}{Theorem}
\newtheorem{claim}[theorem]{Claim}
\newtheorem{corollary}[theorem]{Corollary}
\newtheorem{lemma}[theorem]{Lemma}
\newtheorem{proposition}[theorem]{Proposition}
\newtheorem{conjecture}[theorem]{Conjecture}
\newtheorem{question}[theorem]{Question}

\theoremstyle{definition}
\newtheorem{definition}[theorem]{Definition}
\newtheorem{remark}[theorem]{Remark}

\newcommand\clc{\mathcal{C}}
\newcommand\clp{\mathcal{P}}

\newcommand\clsn{\mathcal{S}_n}
\newcommand{\cl}[1]{\text{cl}(#1)}

\newcommand{\cC}{{\mathfrak C}}
\newcommand{\chain}[2]{\cC(#1,#2)}  
\newcommand{\tch}[3]{\cC^{#3}(#1,#2)} 
\newcommand{\tight}[1]{\,\overset{\scriptstyle #1}{\scriptstyle <}\,}
\newcommand{\cS}{{\mathcal S}}  
                                        
\newcommand{\one}{\mathbf{1}}  
\newcommand{\zero}{\emptyset}  
\newcommand{\suf}[2]{#1_{>#2}} 
\newcommand{\pre}[2]{#1_{\le#2}} 
\newcommand{\ppr}[1]{\pre{\pi}{#1}} 
\newcommand{\psu}[1]{\suf{\pi}{#1}}   
\newcommand{\spr}[1]{\pre{\sigma}{#1}} 
\newcommand{\ssu}[1]{\suf{\sigma}{#1}}
\newcommand{\bsu}{\suf{\beta}{1}} 
\newcommand{\sgn}{\text{sgn}}

\newcommand{\msp}{\ensuremath{\mu(\sigma,\pi)}}

\newcommand{\bmu}{\overline{\mu}} 

\newcommand{\cO}{{\mathcal O}} 

\title[The M\"obius function of separable and decomposable
  permutations]{The M\"obius function of separable and decomposable
  permutations}

\author{Alexander Burstein}
\address{Department of Mathematics, Howard University, Washington DC
  20059, USA}
\email{aburstein@howard.edu}
\thanks{}

\author{V\'it Jel\'inek}
\address{Fakult\"at f\"ur Mathematik, Universit\"at Wien, 
Garnisongasse 3, 1090 Wien, Austria}
\email{jelinek@kam.mff.cuni.cz}
\thanks{Jel\'inek and Steingr\'imsson were supported by grant no.\
  090038012 from the Icelandic Research Fund. Jel\'\i nek was also supported by
grant Z130-N13 from the Austrian Science Foundation (FWF)}

\author{Eva Jel\'inkov\'a}
\address{Department of Applied Mathematics, Charles University 
in Prague, Malostransk\'e n\'am. 25, 110 00, Prague, Czech Republic}
\email{eva@kam.mff.cuni.cz}
\thanks{Jel\'inkov\'a was supported by project 1M0021620838 of the Czech
Ministry of Education}

\author{Einar Steingr\'imsson}%
\address{Department\ of Computer and Information Sciences, University of
  Strathclyde, Glasgow G1 1XH, UK}
\email{einar@alum.mit.edu}%
\thanks{}


\begin{document}
\begin{abstract}
  We give a recursive formula for the M\"obius function of an interval
  $[\sigma,\pi]$ in the poset of permutations ordered by pattern
  containment in the case where $\pi$ is a decomposable permutation,
  that is, consists of two blocks where the first one contains all the
  letters $1,2,\ldots,k$ for some $k$.  This leads to many special
  cases of more explicit formulas.  It also gives rise to a 
  computationally efficient formula for the M\"obius function in the
  case where $\sigma$ and $\pi$ are separable permutations.  A
  permutation is separable if it can be generated from the permutation
  1 by successive sums and skew sums or, equivalently, if it avoids
  the patterns 2413 and 3142.  A consequence of the formula is that
  the M\"obius function of such an interval $[\sigma,\pi]$ is bounded
  by the number of occurrences of $\sigma$ as a pattern in $\pi$. We
  also show that for any separable permutation $\pi$ the M\"obius
  function of $(1,\pi)$ is either 0, 1 or $-1$.
\end{abstract}

\maketitle
\thispagestyle{empty}

\section{Introduction}

Let $\clsn$ be the set of permutations of the integers
$\{1,2,\dots,n\}$.  The union of all $\clsn$ for $n=1,2,\dots$ forms a
poset $\clp$ with respect to \emph{pattern containment}.  That is, we
define $\sigma\le \pi$ in $\clp$ if there is a subsequence of $\pi$
whose letters are in the same order of size as the letters in
$\sigma$.  For example, $132\le24153$, because 2,5,3 appear in the
same order of size as the letters in $132$.  We denote the number of
\emph{occurrences} of $\sigma$ in $\pi$ by $\sigma(\pi)$, for example
$132(24153)=3$, since 243, 253 and 153 are all the occurrences of the
pattern 132 in 24153.

A classical question to ask for any combinatorially defined poset is
what its M\"obius function is.  For our poset $\clp$ this seems to
have first been mentioned explicitly by Wilf \cite{wilf}.  The first
result in this direction was given by Sagan and Vatter
\cite{sagan-vatter}, who showed that an interval $[\sigma,\pi]$ of
\emph{layered} permutations is isomorphic to a certain poset of
compositions of an integer, and they gave a formula for the M\"obius
function in this case. A permutation is layered if it is the
concatenation of decreasing sequences, such that the letters in each
sequence are smaller than all letters in subsequent sequences.
Further results were given by Steingr\'imsson and Tenner
\cite{ste-tenner}, who showed that the M\"obius function
$\mu(\sigma,\pi)$ is 0 whenever the complement of the occurrences of
$\sigma$ in $\pi$ contains an interval block, that is, when $\pi$ has
a segment of two or more consecutive letters that form a segment of
values, where none of these consecutive letters belongs to any
occurrence of $\sigma$ in $\pi$.  One example of such a pair is
$(132,598342617)$, where the letters $342$ do not belong to any
occurrence of $132$ in $598342617$.  Steingr\'imsson and Tenner
\cite{ste-tenner} also described certain intervals where the M\"obius
function is either 1 or $-1$.

In this paper, we focus on permutations that can be expressed as
direct sums or skew sums of smaller permutations. A \emph{direct sum}
of two permutations $\alpha$ and $\beta$, denoted by $\alpha+\beta$,
is the concatenation $\alpha\beta'$, where $\beta'$ is obtained by
incrementing each element of $\beta$ by $|\alpha|$. For example,
31426587 can be written as a direct sum 3142+2143. Similarly, a
\emph{skew sum} $\alpha*\beta$ is the concatenation $\alpha'\beta$
where $\alpha'$ is obtained by incrementing $\alpha$ by $|\beta|$.

A permutation that can be written as a direct sum of two
nonempty permutations is \emph{decomposable}. \emph{The decomposition} of
a permutation $\pi$ is an expression $\pi=\pi_1+\pi_2+\dotsb+\pi_k$ in which
each summand $\pi_i$ is indecomposable. A permutation is \emph{separable} if it
can be obtained from the singleton permutation $\one$ by iterating direct sums
and skew sums (for an alternative definition see Section~\ref{sec-defs}).

Our main result is a set of recurrences for computing the M\"obius
function $\msp$ when $\pi$ is decomposable. If $\pi_1+\dotsb+\pi_k$ is
the decomposition of $\pi$, then these recurrences express $\msp$ in
terms of M\"obius functions involving the summands~$\pi_i$.

In the special case when $\pi$ is separable, these recurrences provide
a polynomial-time algorithm to compute $\mu(\sigma,\pi)$.  These
recurrences also allow us to obtain an alternative combinatorial
interpretation of the M\"obius function of separable permutations,
based on the concept of `normal embeddings'. This interpretation of
$\mu$ generalizes previous results of Sagan and
Vatter~\cite{sagan-vatter} for layered permutations.

Using these expressions of the M\"obius function in terms of normal
embeddings, we derive several bounds on the values of
$\mu(\sigma,\pi)$ for $\sigma$ and $\pi$ separable. In
\cite{ste-tenner}, Steingr\'imsson and Tenner conjectured that for
permutations $\sigma$ and $\pi$ avoiding the pattern 132 (or any one
of the patterns 213, 231, 312) the absolute value of the M\"obius
function of the interval $[\sigma,\pi]$ is bounded by the number of
occurrences of $\sigma$ in $\pi$. We prove this conjecture for the
more general class of separable permutations (for arbitrary $\sigma$
and $\pi$ this bound does not hold in general). In particular, if
$\pi$ has a single occurrence of $\sigma$ then $\mu(\sigma,\pi)$ is
either 1, 0 or $-1$.  We also prove a generalization of another
conjecture mentioned in \cite{ste-tenner}, showing that for any
separable permutation~$\pi$, $\mu(\one,\pi)$ is either 1, 0 or $-1$.

For a non-separable decomposable permutation $\pi$, our recurrences are not
sufficient to compute the value of $\mu(\sigma, \pi)$. Nevertheless, 
they allow us to give short simple formulas in many special cases.

For instance, suppose that $\sigma$ is indecomposable and that $\pi$
is decomposable and of length at least 3.  Then we show that
$\mu(\sigma,\pi)$ can only be nonzero if all the blocks in the
decomposition of $\pi$ are equal to the same permutation $\pi'>\one$,
except possibly the first and the last block, which may be equal
to~$\one$.  In such cases, $\mu(\sigma,\pi)$ equals
$(-1)^i\mu(\sigma,\pi')$, where $i\in\{0,1,2\}$ is the number of
blocks of $\pi$ that are equal to~$\one$.

As another simple example, our results imply that when $\sigma$ and $\pi$ are
permutations with decompositions $\sigma=\sigma_1+\sigma_2$ and
$\pi=\pi_1+\pi_2$, with $\pi_1$ and $\pi_2$ both different from~$\one$, then
$\mu(\sigma,\pi)=\mu(\sigma_1,\pi_1)\mu(\sigma_2,\pi_2)$ if $\pi_1\neq \pi_2$,
and $\mu(\sigma,\pi)=\mu(\sigma_1,\pi_1)\mu(\sigma_2,\pi_2)+\mu(\sigma,\pi_1)$
if $\pi_1=\pi_2$.

The paper is organized as follows: In the next section we provide
necessary definitions.  In Section~\ref{sec-main} we present the main
results, the recursive formulas for reducing the computation of the
M\"obius function of decomposable permutations to that of
indecomposable permutations.  Section~\ref{sec-sep} deals with the
case of separable permutations and their normal embeddings.  Finally,
in Section~\ref{sec-last} we mention some open problems, in particular
questions about the topology of the order complexes of intervals in
our poset, which we have not dealt with in the present paper.

\section{Definitions and Preliminaries}\label{sec-defs}

An \emph{interval} $[\sigma,\pi]$ in a poset $(\clp,\le)$ is the set $\{\rho\colon
\sigma\le \rho\le \pi\}$. In this paper, we deal exclusively with intervals of the
poset of permutations ordered by pattern containment.

The M\"obius function $\mu(\sigma,\pi)$ of an interval $[\sigma,\pi]$
is uniquely defined by setting $\mu(\sigma,\sigma)=1$ for all $\sigma$
and requiring that
\begin{equation}
  \label{eq:eq-mob}
\sum_{\rho\in[\sigma,\pi]}{\mu(\sigma,\rho)}=0
\end{equation}
for every $\sigma<\pi$. 
When $\sigma\not\leq \pi$, we define $\mu(\sigma,\pi)$
to be zero. 

An equivalent definition is given by Philip Hall's Theorem
\cite[Proposition 3.8.5]{ECI}, which says that
\begin{equation}
  \label{eq:eq-chains}
\mu(\sigma,\pi)= \sum_{C\in\chain \sigma\pi}{(-1)^{L(C)}} = \sum_{i}{(-1)^ic_i},
\end{equation}
where $\chain \sigma \pi$ is the set of chains in $[\sigma,\pi]$ that
contain both $\sigma$ and $\pi$, $L(C)$ denotes the length of the
chain $C$, and $c_i$ is the number of such chains of length $i$ in
$[\sigma,\pi]$.  A chain of length $i$ in a poset is a set of $i+1$
pairwise comparable elements $x_0<x_1<\cdots<x_i$.  For details and
further information, see \cite{ECI}.

The \emph{direct sum}, $\alpha+\beta$, of two nonempty permutations
$\alpha$ and $\beta$ is the permutation obtained by concatenating
$\alpha$ and $\beta'$, where $\beta'$ is $\beta$ with all letters
incremented by the number of letters in $\alpha$. A permutation that
can be written as a direct sum of two non-empty permutations is
\emph{decomposable}, otherwise it is \emph{indecomposable}.  Examples
are $2314576= 231+12+21$, and 231, which is indecomposable.  In the
\emph{skew sum} of $\alpha$ and $\beta$, denoted by $\alpha*\beta$, we
increment the letters of $\alpha$ by the length of $\beta$ to obtain
$\alpha'$ and then concatenate $\alpha'$ and $\beta$.  For example,
$6743512 = 12*213*12$. We say that a permutation is
\emph{skew-indecomposable} if it cannot be written as a skew sum of
smaller permutations.

A \emph{decomposition} of $\pi$ is an expression
$\pi=\pi_1+\pi_2+\dotsb+\pi_k$ in which each summand $\pi_i$ is
indecomposable. The summands $\pi_1,\dotsc,\pi_k$ will be called
\emph{the blocks} of~$\pi$. Every permutation $\pi$ has a unique
decomposition (including an indecomposable permutation $\pi$, whose
decomposition has a single block~$\pi$).

A permutation is \emph{separable} if it can be generated from the
permutation~$\one$ by iterated sums and skew sums. In other words, a
permutation is separable if and only if it is equal to $\one$ or it
can be expressed as a sum or skew sum of separable permutations.

Being separable is equivalent to avoiding the patterns 2413 and 3142,
that is, containing no occurrences of either. Separable permutations
have nice algorithmic properties. For instance, Bose, Buss and
Lubiw~\cite{bbl} have shown that it can be decided in polynomial time
whether $\sigma\le \pi$ when $\sigma$ and $\pi$ are separable, while
for general permutations the problem is NP-hard.

It is sometimes convenient to allow permutations to have zero length,
while in other situations, the permutations are assumed to be
nonempty. The unique permutation of length 1 is denoted by $\one$, and
the unique permutation of length 0 is denoted by $\zero$. We make it a
convention that the permutation $\zero$ is neither decomposable nor
indecomposable.  In other words, whenever we say that a permutation
$\pi$ is decomposable (or indecomposable), we automatically assume
that $\pi$ is nonempty.

Suppose that $\pi$ is a nonempty permutation with decomposition
$\pi_1+\dotsb+\pi_n$. For an integer $i\in\{0,\dotsc,n\}$, we let
$\ppr i$ denote the sum $\pi_1+\pi_2+\dotsb+\pi_i$ and let $\psu i$
denote the sum $\pi_{i+1}+\dotsb+\pi_n$. An empty sum of permutations
is assumed to be equal to $\zero$, so in particular $\ppr 0=\psu
n=\zero$. Any permutation of the form $\ppr i$ for some $i$ will be
called a \emph{prefix} of $\pi$, and any permutation of the form $\psu
i$ is a \emph{suffix} of~$\pi$.  Note that $\mu(\zero,\zero)=1$,
$\mu(\zero,\one)=-1$, and it is easily seen that $\mu(\zero,\tau)=0$
for any $\tau>\one$.
                        
\section{The Main Results}\label{sec-main}
Throughout this section, we assume that $\sigma$ is a nonempty
permutation with decomposition $\sigma_1+\dotsb+\sigma_m$ and that
$\pi=\pi_1+\dotsb+\pi_n$ is a decomposable permutation (so $n\ge2$
and, in particular, $\pi$ is nonempty).  The goal in this section is
to prove a set of recurrences that allow us to express the M\"obius
function $\mu(\sigma,\pi)$ in terms of the values of the form
$\mu(\sigma',\pi')$, where $\pi'\in\{\pi_1,\dotsc,\pi_n\}$ is a block
of~$\pi$ and $\sigma'$ is a sum of consecutive blocks
of~$\sigma$. Note that $\sigma$ may itself be indecomposable, in which
case $m$ is equal to $1$ and $\sigma_1=\sigma$.

There are two main recurrences to prove, dealing respectively with the
cases $\pi_1=\one$ and $\pi_1>\one$.

\begin{proposition}[First Recurrence]\label{pro-first}
  Let $\sigma$ and $\pi$ be nonempty permutations with decompositions
  $\sigma=\sigma_1+\dotsb+\sigma_m$ and $\pi=\pi_1+\dotsb+\pi_n$,
  where $n\ge2$.  Suppose that $\pi_1=\one$. Let $k\ge 1$ be the
  largest integer such that all the blocks $\pi_1,\dotsc,\pi_k$ are
  equal to $\one$, and let $\ell\ge 0$ be the largest integer such
  that all the blocks $\sigma_1,\dotsc,\sigma_\ell$ are equal
  to~$\one$. Then
\[
\mu(\sigma,\pi)=
\begin{cases}  
0&\text{ if }k-1>\ell,\\ -\mu(\ssu{k-1},\psu k)&\text{ if
}k-1=\ell,\\ \mu(\ssu k,\psu k)-\mu(\ssu{k-1},\psu k)&\text{ if
}k-1<\ell.
\end{cases}
\]
\end{proposition} 

Note that the suffixes $\ssu{k-1}$, $\ssu{k}$ and $\psu k$ in the
statement of Proposition~\ref{pro-first} may be empty.  This first
recurrence shows how to compute the M\"obius function when $\pi$
starts with $123\ldots k$ for some $k\ge1$.  The second recurrence
takes care of the remaining cases, that is, when $\pi$ does not start
with $1$.

\begin{proposition}[Second Recurrence]\label{pro-second}
  Let $\sigma$ and $\pi$ be nonempty permutations with decompositions
  $\sigma=\sigma_1+\dotsb+\sigma_m$ and $\pi=\pi_1+\dotsb+\pi_n$,
  where $n\ge2$.  Suppose that $\pi_1>\one$. Let $k\ge 1$ be the
  largest integer such that all the blocks $\pi_1,\dotsc,\pi_k$ are
  equal to $\pi_1$. Then
\begin{equation}
  \mu(\sigma,\pi)=\sum_{i=1}^m \sum_{j=1}^k \mu(\spr i,\pi_1)\mu(\ssu
  i, \psu j)\label{eq-second}.
\end{equation}
\end{proposition}

Note that Propositions~\ref{pro-first} and~\ref{pro-second} remain
true when all the direct sums are replaced with skew sums, and the
decompositions are replaced with skew decompositions. To see this, it
is enough to observe that if $\bar\pi$ denotes the reversal of $\pi$
(i.e., $\bar\pi$ is the permutation obtained by reversing the order of
elements of $\pi$), then $\mu(\sigma,\pi)=\mu(\bar\sigma,\bar\pi)$ for
any $\sigma$ and $\pi$, since $[\sigma,\pi]$ and
$[\bar\sigma,\bar\pi]$ are isomorphic posets.

Before we prove the above two recurrences, we give three corollaries
to provide some idea of how the second recurrence can be used.  When
we write $k\times\pi$ we mean a sum $\pi+\pi+\cdots+\pi$ with $k$
summands.

\begin{corollary}\label{coro-first-recurrence}
  Let $\sigma$, $\pi$ and $k$ be as in Proposition~\ref{pro-second},
  and suppose that $\sigma$ is indecomposable, that is, $m=1$.  Then
$$\mu(\sigma,\pi)=
\begin{cases}
\mu(\sigma,\pi_1), & \text{if~} \pi=k\times\pi_1\\ -\mu(\sigma,\pi_1),
& \text{if~} \pi=k\times\pi_1+\one\\ 0, & \text{otherwise}
\end{cases}
$$
\end{corollary}

\begin{proof}
  Since $m=1$, Equation \eqref{eq-second} takes the form
\begin{align*}
  \mu(\sigma,\pi)&=\sum_{j=1}^k \mu(\spr 1,\pi_1)\mu(\ssu 1, \psu
  j)\\ &=\sum_{j=1}^k \mu(\sigma,\pi_1)\mu(\zero, \psu j)\\ &=
  \mu(\sigma,\pi_1)\mu(\zero, \psu k),
\end{align*}                           
where the last equality follows from the fact that $\mu(\zero,\psu j)$
is equal to 0 whenever $\psu j$ has more than one block.
                
We have $\mu(\zero, \psu k)=1$ when $\psu k=\zero$, $\mu(\zero, \psu
k)=-1$ when $\psu k=\one$, and $\mu(\zero, \psu k)=0$ otherwise. In
particular, $\mu(\zero, \psu k)$ can only be nonzero either when $k=n$
and $\pi=k\times\pi_1$, or when $k=n-1$ and $\pi=k\times\pi_1+\one$.
\end{proof}

Corollary \ref{coro-first-recurrence} implies that if $\sigma$ is
indecomposable and $\pi$ is decomposable, then almost always
$\mu(\sigma,\pi)=0$, since the two exceptions for $\pi$ given in the
corollary are of a proportion that clearly goes to zero among
decomposable permutations as their length goes to infinity.

\begin{corollary}
  With $\sigma$ and $\pi$ as in Proposition~\ref{pro-second}, assume
  that $\sigma$ and $\pi$ decompose into exactly two blocks, with
  $\sigma=\sigma_1+\sigma_2$ and $\pi=\pi_1+\pi_2$, and that $\pi_1,
  \pi_2>\one$. Then
$$\mu(\sigma,\pi)=
\begin{cases}
  \mu(\sigma_1,\pi_1)\mu(\sigma_2,\pi_2), & \text{if~}
  \pi_1\ne\pi_2\\ \mu(\sigma_1,\pi_1)\mu(\sigma_2,\pi_2)+\mu(\sigma,\pi_1)
  & \text{if~} \pi_1=\pi_2
\end{cases}
$$
\end{corollary}

\begin{proof}
If $\pi_1\neq \pi_2$ (so $k=1$), then the summation in
Equation~\eqref{eq-second} expands into
\begin{align*}
  \mu(\sigma,\pi)=\mu(\sigma_1,\pi_1)\mu(\sigma_2,\pi_2)+\mu(\sigma,\pi_1)\mu(\zero,\pi_2).
\end{align*}
Since $\pi_2>\one$, the second summand vanishes and
$\mu(\sigma,\pi)=\mu(\sigma_1,\pi_1)\mu(\sigma_2,\pi_2)$.

If, on the other hand, $\pi_1=\pi_2$, then Equation~\eqref{eq-second}
states that $\mu(\sigma,\pi)$ is equal to
\begin{align*} 
  &\mu(\sigma_1,\pi_1)\mu(\sigma_2,\pi_2)+\mu(\sigma_1,\pi_1)\mu(\sigma_2,\zero)+
  \mu(\sigma,\pi_1)\mu(\zero,\pi_2)+\mu(\sigma,\pi_1)\mu(\zero,\zero)\\ =&\mu(\sigma_1,\pi_1)\mu(\sigma_2,\pi_2)+\mu(\sigma,\pi_1).\qedhere
\end{align*}
\end{proof}

\begin{remark}
  An obvious question to ask is whether the product formula
  $\mu(\sigma,\pi)=\mu(\sigma_1,\pi_1)\mu(\sigma_2,\pi_2)$, in the
  case when $\pi_1\ne\pi_2$, is a result of the interval
  $(\sigma,\pi)$ being isomorphic to the direct product of the
  intervals $[\sigma_1,\pi_1]$ and $[\sigma_2,\pi_2]$.  Although this
  seems to occur frequently, it does not hold in general.
\end{remark}

The following corollary is an immediate consequence of Proposition
\ref{pro-first} (the case when $k-1=\ell=0$).

\begin{corollary}\label{coro-not-high}
Suppose $\sigma$ and $\pi$ are permutations of length at least two,
such that neither begins with 1.  Then $\mu(\sigma,\one+\pi)=-\msp$.
\end{corollary}

Both recurrences (Propositions \ref{pro-first} and \ref{pro-second})
are proved using arguments that involve cancellation between certain
types of chains in the poset of permutations. Let us first introduce
some useful notation. For a chain
$C=\{\alpha_0<\alpha_1<\dotsb<\alpha_k\}$ of permutations let $L(C)$
denote the length of $C$, which is one less than the number of
elements of~$C$. The \emph{weight of $C$}, denoted by $w(C)$, is the
quantity $(-1)^{L(C)}$. If $\cC$ is any set of chains, then the
\emph{weight of $\cC$} is defined by
\[
w(\cC)=\sum_{C\in \cC} w(C)=\sum_{C\in \cC} (-1)^{L(C)}.
\]

Recall that $\chain{\sigma}{\pi}$ is the set of all the chains from
$\sigma$ to~$\pi$ that contain both $\sigma$ and~$\pi$. We know that
$\mu(\sigma,\pi)=w(\chain{\sigma}{\pi})$, by Philip Hall's Theorem.

For a chain $C=\{\alpha_0<\alpha_1<\dotsb<\alpha_k\}$ and a permutation 
$\beta$, we let $\beta+C$ denote the chain $\{ 
\beta+\alpha_0<\beta+\alpha_1<\dotsb<\beta+\alpha_k\}$. The chain $C+\beta$ is 
defined analogously.
                   
\subsection{Proof of the first recurrence}

Let us now turn to the proof of Proposition~\ref{pro-first}. Suppose
that $\sigma$, $\pi$, $m$, $n$, $k$, and $\ell$ are as in the
statement of the proposition. For a permutation $\tau\in\cS$, define
the \emph{degree of $\tau$}, denoted by $\deg(\tau)$, to be the
largest integer $d$ such that $\tau$ can be expressed as
$d\times\one+\tau'$ for some (possibly empty) permutation~$\tau'$. In
particular, we have $k=\deg(\pi)$ and $\ell=\deg(\sigma)$.

Let $C=\{\tau_0<\tau_1<\dotsb<\tau_p\}$ be a chain of permutations. We say 
that a permutation $\tau_i\in C$ is the \emph{pivot} of the chain $C$, if 
$\deg(\tau_i)<\deg(\tau_j)$ for each $j>i$, and $\deg(\tau_i)\le 
\deg(\tau_j)$ for each $j\le i$. In other words, the pivot is the element of 
the chain with minimum degree, and if there are more elements of minimum 
degree, the pivot is the largest of them.

Let $\rho$ denote the permutation~$\psu 1$. Obviously $\deg(\rho)=k-1$ and 
$\one+\rho=\pi$. We partition the set of chains $\chain{\sigma}{\pi}$ into 
three disjoint subsets, denoted by $\cC_a$, $\cC_b$ and $\cC_c$, and we compute
the weight of each subset separately. A chain $C\in\chain{\sigma}{\pi}$ belongs
to $\cC_a$ if its pivot is the permutation~$\pi$, the chain $C$ belongs to
$\cC_b$ if its pivot is the permutation $\rho$, and $C$ belongs to $\cC_c$
otherwise. We now separate the main steps of the proof into independent claims.
                                                                             
\begin{claim}\label{cla-c1}
If $\deg(\sigma)<\deg(\pi)$ (so $\ell<k$), then $\cC_a$ is empty.
Otherwise, $w(\cC_a)=\mu(\ssu k, \psu k)$.
\end{claim} 
\begin{proof} Obviously, if $\deg(\sigma)<\deg(\pi)$, then no chain
  from $\sigma$ to $\pi$ can have $\pi$ for pivot, because the pivot
  must have minimal degree among the elements of the chain. Thus,
  $\cC_a$ is empty.

Assume now that $\deg(\sigma)\ge \deg(\pi)$. We show that there is a 
length-preserving bijection between the set of chains $\chain{\ssu k}{\psu k}$ 
and the set of chains $\cC_a$. Indeed, take any chain $C\in \chain{\ssu k}{\psu 
k}$, and create a new chain $f(C)=(k\times\one)+C$.
Then $f(C)$ is a chain from $\sigma$ to $\pi$, and since every element of 
$f(C)$ has degree at least $k$, while $\pi$ has degree exactly $k$, we see 
that $\pi$ is the pivot of~$f(C)$. Hence $f(C)\in\cC_a$. 

On the other hand, if $C'$ is any chain from $\cC_a$, we see that each 
element of $C'$ has degree at least $k$, because $\pi$ has degree $k$ and is 
the pivot of~$C'$. Thus, every element $\tau'\in C'$ is of the form 
$k\times\one+\tau$ for some $\tau$, and hence there exists a chain 
$C\in\chain{\ssu k}{\psu k}$ such that $C'=f(C)$. Since $f$ is clearly 
injective and length-preserving, we conclude that $w(\cC_a)=w(\chain{\ssu k}{\psu 
k})=\mu(\ssu k,\psu k)$, as claimed.
\end{proof} 

\begin{claim}\label{cla-c2} 
If $\deg(\sigma)<\deg(\rho)$ (so $\ell<k-1$), then $\cC_b$ is empty.
Otherwise, $w(\cC_b)=-\mu(\ssu{k-1},\psu{k})$.
\end{claim}                               
\begin{proof}
If $\deg(\sigma)<\deg(\rho)$ then $\rho$ cannot be the pivot of any chain 
containing $\sigma$ and $\cC_b$ is empty.

Assume now that $\deg(\sigma)\ge\deg(\rho)$. We will describe a 
parity-reversing bijection $f$ between the set of chains 
$\chain{\ssu{k-1}}{\psu{k}}$ and the set of chains $\cC_b$. Take a chain 
$C\in\chain{\ssu{k-1}}{\psu{k}}$. Define a new chain $C'$ by 
$C'=((k-1)\times\one)+C$. Notice that $C'$ is a chain from 
$\sigma$ to $\rho$ whose pivot is~$\rho$ and whose length is equal to the 
length of~$C$. Define the chain $f(C)$ by $f(C)=C'\cup\{\pi\}$. Then the 
chain $f(C)$ belongs to $\cC_b$ and has length $L(C)+1$. It is again easy 
to see that $f$ is a bijection between $\chain{\ssu{k-1}}{\psu{k}}$ and 
$\cC_b$, which shows that
\[
w(\cC_b)=-w(\chain{\ssu{k-1}}{\psu{k}})=-\mu(\ssu{k-1},\psu k),
\]                                                            
as claimed.
\end{proof}

\begin{claim}\label{cla-c3}
$w(\cC_c)=0$.
\end{claim}

\begin{proof} We construct a parity-reversing involution
  $f\colon\cC_c\to\cC_c$. Let $C$ be a chain from $\cC_c$, let $\tau$
  be its pivot, and let $\tau'$ be the successor of $\tau$ in~$C$. By
  definition of $\cC_c$, $\tau$ is not equal to $\pi$, so $\tau'$ is
  well defined. From the definition of a pivot, we know that
  $\deg(\tau)<\deg(\tau')$.  Let us distinguish two cases:
\begin{enumerate}
\item
If $\tau'=\one+\tau$, we define a new chain 
$f(C)$ by $f(C)=C\setminus\{\tau'\}$. Note that in this case, we know that 
$\tau'$ is different from $\pi$, because otherwise $\tau$ would be equal to 
$\rho$, contradicting the definition of $\cC_c$. Thus, $f(C)\in\cC_c$. Note 
that $\tau$ is a pivot of~$f(C)$.
\item If $\tau'\neq \one+\tau$, then we easily deduce that
  $\tau'>\one+\tau$ (recall that $\deg(\tau')>\deg(\tau)$). We then
  define a new chain $f(C)=C\cup\{\one+\tau\}$, in which the new
  element $\one+\tau$ is inserted between $\tau$ and $\tau'$.
\end{enumerate} 
The mapping $f$ is easily seen to be an involution on the set $\cC_c$ which 
preserves the pivot and maps odd-length chains to even-length chains and vice 
versa. This shows that $w(\cC_c)=0$, as claimed.
\end{proof} 

From these claims, Proposition~\ref{pro-first} easily follows. Indeed, 
Claim~\ref{cla-c3} implies that $\mu(\sigma,\pi)=w(\cC_a)+w(\cC_b)$. From 
Claims~\ref{cla-c1} and \ref{cla-c2} we deduce the values of 
$\mu(\sigma,\pi)$:
\begin{itemize}   
\item
If $k-1>\ell$ then both $\cC_a$ and $\cC_b$ are empty and $\mu(\sigma,\pi)=0$.
\item 
If $k-1=\ell$ then $\cC_a$ is empty and 
$\mu(\sigma,\pi)=w(\cC_b)=-\mu(\ssu{k-1},\psu{k})$.
\item 
If $k-1<\ell$, then $\mu(\sigma,\pi)=w(\cC_a)+w(\cC_b)=\mu(\ssu k, \psu 
k)-\mu(\ssu{k-1},\psu{k})$.
\end{itemize}
This completes the proof of Proposition~\ref{pro-first}.    

\subsection{Proof of the second recurrence}

It remains to prove Proposition~\ref{pro-second}. The proof is again based on 
cancellation among the chains from $\sigma$ to~$\pi$. Before stating the proof, we 
need more terminology and several lemmas.

Let $\alpha$, $\beta$ and $\rho$ be any permutations. We say that $\alpha$ is 
a \emph{$\rho$-tight subpermutation of $\beta$}, denoted by 
$\alpha\tight{\rho}\beta$, if $\alpha<\beta$ but $\rho+\alpha$ is not 
contained in~$\beta$. We say that a chain 
$\{\alpha_0<\alpha_1<\dotsb<\alpha_k\}$ is \emph{$\rho$-tight} if 
$\alpha_{i-1}\tight{\rho}\alpha_i$ for every $i=1,\dotsc,k$. Let 
$\tch{\alpha}{\beta}{\rho}$ be the set of all the $\rho$-tight chains from 
$\alpha$ to~$\beta$.   

The following simple properties of $\rho$-tightness follow directly from the 
definitions, and they are presented without proof. 

\begin{lemma}\label{obs-tight}
For arbitrary permutations $\alpha$, $\beta$, $\gamma$ and $\rho$,
we have $\alpha+\gamma\tight{\rho}\beta+\gamma$ if and only if 
$\alpha\tight{\rho}\beta$.
\end{lemma}
 
\begin{lemma}\label{obs-tight2}
If $\rho$ is a nonempty indecomposable permutation, and if $\alpha$ and 
$\beta$ are arbitrary permutations, then $\rho+\alpha\tight{\one}\rho+\beta$ 
if and only if $\alpha\tight{\rho}\beta$.
\end{lemma}

The next lemma shows the relevance of $\rho$-tightness for the computation
of~$\mu$.

\begin{lemma}\label{lem-tight} Let $\beta$ be a permutation with decomposition 
$\beta=\beta_1+\beta_2+\dotsb+\beta_p$.  Let $\rho$ be a 
nonempty indecomposable permutation, and let $\alpha$ be any permutation.
\begin{enumerate}
\item
If $\rho\neq\beta_1$, then $\mu(\alpha,\beta)=w(\tch{\alpha}{\beta}{\rho})$.
\item
If $\rho=\beta_1$, then 
$\mu(\alpha,\beta)=w(\tch{\alpha}{\beta}{\rho})-w(\tch{\alpha}{\bsu}{\rho})$.
\end{enumerate}
\end{lemma} 
\begin{proof}
Let us first deal with the first claim of the lemma. Let us define 
$\widehat\cC=\chain{\alpha}{\beta}\setminus\tch{\alpha}{\beta}{\rho}$ to be the set 
of all the chains from $\alpha$ to $\beta$ that are not $\rho$-tight. The 
first part of the lemma is equivalent to saying that $w(\widehat\cC)=0$. To prove 
this, we find a parity-reversing involution $f$ on the set~$\widehat\cC$.

Consider a chain 
$C=\{\alpha=\alpha_0<\alpha_1<\dotsb<\alpha_q=\beta\}\in\widehat\cC$. Since $C$ is 
not $\rho$-tight, there is an index $i$ such that $\rho+\alpha_i\le 
\alpha_{i+1}$. Fix the smallest such value of~$i$. We distinguish two cases: 
either $\rho+\alpha_i<\alpha_{i+1}$, or $\rho+\alpha_i=\alpha_{i+1}$. 

If $\rho+\alpha_i<\alpha_{i+1}$, define a new chain 
\[
f(C)=C\cup\{\rho+\alpha_i\}= 
\{\alpha=\alpha_0<\alpha_1<\dotsb<\alpha_i<\rho+\alpha_i<\alpha_{i+1}<\dotsb<\alpha_q=\beta\}.
\]
On the other hand, if $\rho+\alpha_i=\alpha_{i+1}$, define a new chain
\[
f(C)=C\setminus\{\rho+\alpha_i\}= 
\{\alpha=\alpha_0<\alpha_1<\dotsb<\alpha_i<\alpha_{i+2}<\dotsb<\alpha_q=\beta\}.
\]   
Note that, since we assume that $\rho\neq\beta_1$ and that $\rho$ is
indecomposable, we know that $\rho+\alpha_i$ is not equal to~$\beta$.
Moreover, in the chain $f(C)$ the element $\alpha_i$ is not a
$\rho$-tight subpermutation of its successor in the chain. Thus, we
see that $f(C)$ is a chain from~$\widehat\cC$. It is easy to see that
$f$ is an involution, and that it reverses the parity of the length of
the chain, showing that $w(\widehat\cC)=0$.  This proves the first
part of the lemma.
                                                                    
Let us prove the second part. Assume that $\rho=\beta_1$, that is, 
$\beta=\rho+\bsu$. Consider a chain $C$ from $\alpha$ to $\beta$, and let 
$\alpha_0,\alpha_1,\dotsc,\alpha_q$ be the elements of~$C$. We say that the 
chain $C$ is \emph{almost $\rho$-tight} if its second largest element 
$\alpha_{q-1}$ is equal to $\bsu$ and if $\alpha_{i-1}\tight{\rho}\alpha_i$ 
for each~$i\le q-1$. Note that an almost $\rho$-tight chain is never 
$\rho$-tight, because $\bsu$ is not a $\rho$-tight subpermutation of 
$\beta=\rho+\bsu$.
          
We partition the set $\chain{\alpha}{\beta}$ into three disjoint sets 
$\cC_a$, $\cC_b$, and $\cC_c$, where $\cC_a$ is the set 
$\tch{\alpha}{\beta}{\rho}$ of $\rho$-tight chains, $\cC_b$ is the set of 
almost $\rho$-tight chains, and $\cC_c$ contains the chains that  neither 
$\rho$-tight nor almost $\rho$-tight. 

Consider again the mapping $f$ defined in the proof of the first part of the 
lemma. This mapping, restricted to the set $\cC_c$, is easily seen to be a 
parity-reversing involution on~$\cC_c$, which shows that~$w(\cC_c)=0$. This 
means that $\mu(\alpha,\beta)=w(\cC_a)+w(\cC_b)$. 

Furthermore, note that an almost $\rho$-tight chain from $\alpha$ to $\beta$ 
consists of a $\rho$-tight chain from $\alpha$ to $\bsu$ followed by~$\beta$, 
and conversely, any $\rho$-tight chain from $\alpha$ to $\bsu$ can be extended to an 
almost $\rho$-tight chain from $\alpha$ to $\beta$ by adding the 
element~$\beta$. Thus, we see that $w(\cC_b)=-w(\tch{\alpha}{\bsu}{\rho})$. 
This implies that 
$\mu(\alpha,\beta)=w(\tch{\alpha}{\beta}{\rho})-w(\tch{\alpha}{\bsu}{\rho})$, 
as claimed.
\end{proof}
         
The next lemma is an easy consequence of Lemma~\ref{lem-tight}.

\begin{lemma}\label{lem-tight2}
Let $\beta$ be a permutation with decomposition 
$\beta=\beta_1+\beta_2+\dotsb+\beta_p$. Let $\rho$ be an indecomposable 
permutation, and let $\alpha$ be any permutation. Let $q\ge 0$ be the largest 
integer such that the blocks $\beta_1,\beta_2,\dotsc,\beta_q$ are all equal 
to~$\rho$. Then
\[
w(\tch{\alpha}{\beta}{\rho})=\sum_{i=0}^{q} \mu(\alpha,\suf{\beta}{i}).
\]
\end{lemma}  
\begin{proof} Proceed by induction on~$q$. If $q=0$, the claim reduces to the 
identity $w(\tch{\alpha}{\beta}{\rho})=\mu(\alpha,\beta)$, which follows from 
the first part of Lemma~\ref{lem-tight}. Suppose that $q>0$. Then the second 
part of Lemma~\ref{lem-tight} applies and we get that  
\[
\mu(\alpha,\beta)=w(\tch{\alpha}{\beta}{\rho})-w(\tch{\alpha}{\bsu}{\rho}),
\]                                                                         
which is equivalent to 
\begin{equation}
w(\tch{\alpha}{\beta}{\rho})=\mu(\alpha,\beta)+w(\tch{\alpha}{\bsu}{\rho}).\label{eq-tight2}
\end{equation}                                                                         
By induction, we know that 
\[
w(\tch{\alpha}{\bsu}{\rho})=\sum_{i=0}^{q-1} \mu(\alpha,\suf{(\bsu)}{i})=\sum_{i=1}^{q} 
\mu(\alpha,\suf{\beta}{i}).
\]                         
Combining this with~\eqref{eq-tight2}, we obtain the desired identity.
\end{proof}                                                           

Before we proceed towards the proof of Proposition~\ref{pro-second}, we need 
to introduce more definitions. Let $\beta$ be a permutation with 
decomposition $\beta_1+\dotsb+\beta_p$ into indecomposable blocks, let 
$\alpha$ be any permutation. Let $C$ be a chain of permutations, with 
elements $\alpha=\alpha_0<\alpha_1<\dotsb<\alpha_q=\beta$. We express each 
element $\alpha_i$ of the chain as a sum of two permutations, called 
\emph{head} and \emph{tail}, denoted respectively as $h_i(C)$ and $t_i(C)$, 
with $\alpha_i=h_i(C)+t_i(C)$. The head and tail are defined inductively as 
follows: for $i=q$, we have $\alpha_i=\alpha_q=\beta$ and we define
$h_q(C)=\beta_1$ and $t_q(C)=\bsu$. 

Suppose now that the head and tail of $\alpha_i$ have been already defined, 
and let us define head and tail of $\alpha_{i-1}$. Let us put 
$\gamma=\alpha_{i-1}$, and assume that $\gamma$ has decomposition 
$\gamma_1+\gamma_2+\dotsb+\gamma_r$ into indecomposable blocks. Let $j$ be 
the smallest integer such that $\suf{\gamma}{j}\le t_i(C)$. It then follows 
that $\pre{\gamma}{j}\le h_i(C)$. We define $h_{i-1}(C)=\pre{\gamma}{j}$ and 
$t_{i-1}(C)=\suf{\gamma}{j}$. In other words, the tail of $\alpha_{i-1}$ is 
its longest suffix that is contained in the tail of~$\alpha_i$.

If the chain $C$ is clear from the context, we write $h_i$ and $t_i$ instead 
of $h_i(C)$ and $t_i(C)$. Note that $h_0\le h_1\le\dotsb\le h_q$ and $t_0\le 
t_1\le \dotsb\le t_q$.

We say that the chain $C$ of length $q$ is \emph{split} if there is an index 
$s\in\{0,\dotsc,q\}$ such that $t_0=t_1=\dotsb=t_s$ and 
$h_s=h_{s+1}=\dotsb=h_q$. Such an index $s$ is then necessarily unique. 
The next lemma demonstrates the relevance of these notions.

\begin{lemma}\label{lem-split}
Let $\beta$ be a permutation with decomposition 
$\beta_1+\beta_2+\dotsb+\beta_p$ such that $\beta_1\neq\one$. Let 
$\alpha$ be an arbitrary permutation. Let $\cC^{*}$ be the set of all the 
chains from $\chain{\alpha}{\beta}$ which are split and $\one$-tight. 
Then $\mu(\alpha,\beta)=w(\cC^{*})$.
\end{lemma}
\begin{proof} By the first part of Lemma~\ref{lem-tight}, we know that 
$\mu(\alpha,\beta)$ is equal to $w(\tch{\alpha}{\beta}{\one})$, that is, to the 
total weight of all the $\one$-tight chains from $\alpha$ to $\beta$. Define 
the set $\widehat\cC=\tch{\alpha}{\beta}{\one}\setminus\cC^{*}$ of all the 
$\one$-tight, non-split chains from $\alpha$ to~$\beta$. 

To prove the lemma, we need to show that $w(\widehat\cC)=0$. To achieve this, 
we again use a parity-reversing involution $f$ on the set $\widehat\cC$. 
Consider a chain $C\in\widehat\cC$ with elements 
$\alpha_0<\alpha_1<\dotsb<\alpha_q$. Since $C$ is not split, there must exist 
an index $j\in\{1,\dotsc,q\}$ such that either 
\begin{enumerate}
\item\label{case1}
$h_{j-1}<h_j$ and $t_{j-1}<t_j$, or 
\item\label{case2}
$h_{j-1}=h_j<h_{j+1}$ and $t_{j-1}<t_j=t_{j+1}$.
\end{enumerate}
Fix such an index $j$ as large as possible and distinguish two cases depending 
on which of the two above-mentioned possibilities occur for this index~$j$.

\emph{Case} (\ref{case1}). 
Assume that $h_{j-1}<h_j$ and $t_{j-1}<t_j$. Let us write $h=h_{j-1}$, 
$H=h_j$, $t=t_{j-1}$, and $T=t_j$, so we have $\alpha_{j-1}=h+t$ and 
$\alpha_j=H+T$. Define a permutation $\gamma=h+T$, and a new chain 
$f(C)=C\cup\{\gamma\}$. Note that since $C$ is a $\one$-tight chain, and in 
particular $\alpha_{j-1}\tight{\one}\alpha_j$, we also have 
$\alpha_{j-1}\tight{\one}\gamma\tight{\one}\alpha_j$, and hence $f(C)$ is a 
$\one$-tight chain as well. 

We need to prove that $f(C)\in\widehat\cC$, which follows easily from the following 
claim.

\begin{claim}\label{cla-head}
Each permutation of $C$ has the same head and tail in $f(C)$ as in~$C$. The 
permutation $\gamma=h+T$ has head $h$ and tail $T$ in~$f(C)$.
\end{claim}
\begin{proof}[Proof of Claim~\ref{cla-head}]
It is clear that the claim holds for all the permutations that are greater 
than~$\gamma$. 

It is also easy to see that the claim holds for $\gamma$. Indeed, the 
successor of $\gamma$ in $f(C)$ is the permutation $H+T$, whose tail is~$T$. 
Since the tail of $\gamma$ cannot be greater than $T$ and since $\gamma=h+T$, 
it follows that the tail of $\gamma$ is $T$ and its head is~$h$. 

Let us now consider the permutation $\alpha_{j-1}=h+t$. The successor of 
$\alpha_{j-1}$ in $C$ is the permutation $\alpha_j=H+T$, and the successor of 
$\alpha_{j-1}$ in $f(C)$ is the permutation $\gamma=h+T$. Since the two 
successors have the same tail $T$, and since the tail of a permutation only 
depends on the tail of its successor, we see that $\alpha_{j-1}$ has the same
tail (and hence also the same head) in $f(C)$ as in~$C$. 

From these facts, the claim immediately follows.
\end{proof}     

We may now conclude that $f(C)\in \widehat\cC$, and turn to the second case of
the proof of the lemma. 
                                               
\emph{Case} (\ref{case2}).  Assume now that $h_{j-1}=h_j<h_{j+1}$ and 
$t_{j-1}<t_j=t_{j+1}$. Let us define $h=h_{j-1}=h_j$, $H=h_{j+1}$, 
$t=t_{j-1}$, and $T=t_j=t_{j+1}$. In particular, $\alpha_{j-1}=h+t$, 
$\alpha_j=h+T$, and $\alpha_{j+1}=H+T$. Define the chain 
$f(C)=C\setminus\{\alpha_j\}$.                    

We claim that $f(C)$ is $\one$-tight. To see this, it is enough to
prove ${h+t\tight{\one}H+T}$. Assume, for a contradiction, that
$\one+h+t\le H+T$. In any occurrence of $\one+h+t$ inside $H+T$, the
prefix $\one+h$ must occur inside $H$, otherwise we get a
contradiction with the assumption that $t$ is the tail
of~$\alpha_{j-1}$. This shows that $\one+h\le H$, and hence
$\one+h+T=\one+\alpha_j\le \alpha_{j+1}=H+T$, contradicting the
assumption that $C$ is $\one$-tight. To finish the proof of the lemma,
we need one more claim.

\begin{claim}\label{cla-head2}
Each permutation of $f(C)$ has the same head and tail in $f(C)$ as in $C$.
\end{claim}                                                               
\begin{proof}[Proof of Claim~\ref{cla-head2}]
It is enough to prove the claim for the permutation $\alpha_{j-1}=h+t$, because
any other permutation of $f(C)$ has the same successor in $f(C)$ as in~$C$. For 
$\alpha_{j-1}$, the claim follows from the fact that the successor of 
$\alpha_{j-1}$ in $C$ has the same tail as the successor of $\alpha_{j-1}$ 
in~$f(C)$. This completes the proof of the claim. 
\end{proof}
We now see that even in this second case, $f(C)$ belongs to~$\widehat\cC$.

Combining the two cases described above, we see that $f$ is a
parity-reversing involution of the set $\widehat\cC$. This means that
$w(\widehat\cC)=0$, and consequently, $\mu(\alpha,\beta)=w(\cC^{*})$,
as claimed. This completes the proof of the lemma.
\end{proof}        

Finally, we can prove Proposition~\ref{pro-second}. Assume that
$\sigma$ is a permutation with decomposition
$\sigma_1+\dotsb+\sigma_m$ and that $\pi$ is a permutation with
decomposition $\pi_1+\dotsb+\pi_n$, where $n\ge 2$ and
$\pi_1>\one$. Let $k\ge 1$ be the largest integer such that all the
blocks $\pi_1,\dotsc,\pi_k$ are equal to~$\pi_1$. Recall that our goal
is to prove identity \eqref{eq-second}, which reads as follows:
\[
\mu(\sigma,\pi)=\sum_{i=1}^m \sum_{j=1}^k \mu(\spr i,\pi_1)\mu(\ssu i, \psu 
j).
\]                                                     

Let $\cC^{*}$ be the set of $\one$-tight split chains from $\sigma$ to
$\pi$.  From Lemma~\ref{lem-split}, we know that
$\mu(\sigma,\pi)=w(\cC^{*})$. For a chain $C\in\cC^{*}$, let $t_0(C)$
be the tail of the element $\sigma\in C$, which is the smallest
element in the chain. By definition, $t_0(C)$ is a suffix of $\sigma$,
that is, it is equal to $\ssu i$ for some value of
$i\in\{0,\dotsc,m\}$. Define, for each $i\in\{0,\dotsc,m\}$, the set
of chains
\[
\cC_i=\{C\in\cC^{*}, t_0(C)=\ssu i\}.
\]
The sets $\cC_i$ form a disjoint partition of $\cC^{*}$. We will now compute the 
weight of the individual sets~$\cC_i$.

\begin{claim}\label{cla-c0}
Let $C$ be a chain from $\cC^{*}$. Every element of $C$ has nonempty head. 
Consequently, $t_0(C)$ is never equal to $\sigma$, and hence $\cC_0$ is empty.
\end{claim}                                                                   
\begin{proof}
Suppose that $C$ has an element with empty head. Let $\alpha$ be the largest 
such element. By definition, the element $\pi\in C$ has head equal to 
$\pi_1$, so $\alpha\neq \pi$. In particular, $\alpha$ has a successor 
$\alpha'$ in $C$, and $\alpha'$ has nonempty head. Let $h'$ and $t'$ be the 
head and tail of $\alpha'$. By assumption, $h'$ is nonempty, which means that 
$\one\le h'$.  Moreover, $\alpha\le t'$, because $\alpha$ is its own tail. 
This means that $\one+\alpha\le \alpha'$, which is impossible because the 
chains in $\cC^{*}$ are assumed to be $\one$-tight.

This shows that every element of $C$ has nonempty head, and the rest of the 
claim follows directly.
\end{proof}            

Claim~\ref{cla-c0} implies that $w(\cC_0)=0$, and hence 
$\mu(\sigma,\pi)=\sum_{i\ge 1} w(\cC_i)$. It remains to determine the value of 
$w(\cC_i)$ for~$i>0$. 

Fix an integer $i\in\{1,\dotsc,m\}$. Define $h=\spr i$, $t=\ssu i$, 
$H=\pi_1$, and $T=\psu 1$. Note that in a chain $C\in\cC_i$, the 
permutation $\sigma$ has head $h$ and tail $t$, 
while the permutation $\pi$ has head $H$ and 
tail~$T$.

\begin{claim}\label{cla-concat} With the notation as above, 
\[
w(\cC_i)=w(\tch{h}{H}{\one}) w(\tch{t}{T}{H}).
\]                                                                  
\end{claim}    
\newcommand{\cCa}{{\cC'}}
\newcommand{\cCb}{{\cC''}}                                                     
\begin{proof}                         
Let us write $\cCa=\tch{h}{H}{\one}$ and $\cCb=\tch{t}{T}{H}$. We will 
provide a bijection $f\colon \cCa\times\cCb\to\cC_i$, which maps a pair of 
chains $(C_1,C_2)\in \cCa\times\cCb$ to a chain $f(C_1,C_2)\in\cC_i$ whose length 
is equal to $L(C_1)+L(C_2)$. Such a bijection immediately implies the 
identity $w(\cC_i)=w(\cCa)w(\cCb)$ from the claim.

The definition of the mapping $f$ is simple: for $C_1\in\cCa$ and 
$C_2\in\cCb$, define $f(C_1,C_2)$ to be the concatenation of the two chains 
$C_1+t$ and $H+C_2$. This is well defined, since the maximum of $C_1+t$ is 
the permutation $H+t$, which is also equal to the minimum of the chain 
$H+C_2$. Thus, $f(C_1,C_2)$ is a chain of length $L(C_1)+L(C_2)$. Let 
us denote this chain by~$C$.

We now show that $C$ belongs to $\cC_i$. Let us call the two sub-chains 
$C_1+t$ and $H+C_2$ respectively the \emph{bottom part} and the \emph{top 
part} of~$C$. Note that the permutation $H+t$ is the unique element of $C$ 
belonging both to the top part and to the bottom part.

By construction, $C$ is a chain from $\sigma$ to $\pi$. The bottom part of 
$C$ is a $\one$-tight chain, because $C_1$ was assumed to be $\one$-tight 
(see Lemma~\ref{obs-tight}). Similarly, by 
Lemma~\ref{obs-tight2}, the top part of $C$ is a $\one$-tight chain, 
because $C_2$ is $H$-tight and $H$ is indecomposable. This shows that the 
chain $C$ is $\one$-tight.

Our next step is to prove that every element in the top part of $C$ has head 
equal to $H$, and that every element in the bottom part of $C$ has tail equal 
to~$t$. Assume that this statement is false, and let $\alpha$ be the largest 
element of $C$ for which it fails. Clearly, $\alpha\neq\pi$, so $\alpha$ has 
a successor $\beta$ in $C$. Suppose first that $\alpha$ belongs to the top 
part of $C$. Then $\alpha$ can be written as a sum $H+\alpha'$ for some 
$\alpha'\in C_2$, and likewise $\beta=H+\beta'$ for $\beta'\in C_2$. By the choice of 
$\alpha$, we know that the head of $\beta$ is $H$ and hence its tail is 
$\beta'$. Since $\alpha'<\beta'$, the tail of $\alpha$ contains~$\alpha'$. On 
the other hand, the only suffix of $\alpha$ longer than $\alpha'$ is the 
permutation $\alpha$ itself, because $H$ is indecomposable. By 
Claim~\ref{cla-c0}, the head of $\alpha$ must be nonempty, which means that 
the head of $\alpha$ can only be equal to $H$, which contradicts our choice 
of~$\alpha$.

Suppose now that $\alpha$ does not belong to the top part of $C$. Then 
$\beta$ belongs to the bottom part of $C$ (and possibly to the top part as 
well). Consequently, $\alpha$ can be written as $\alpha'+t$ and $\beta$ can 
be written as $\beta'+t$, with $\alpha',\beta'\in C_1$. We also know that $t$ 
is the tail of $\beta$. This makes it clear that $t$ is the tail of $\alpha$ 
as well, which is a contradiction. 

This proves that all the elements of the top part of $C$ indeed have head 
$H$, and all the elements in the bottom part have tail $t$. This shows that 
$C$ is a split chain and also that $t_0(C)=t$. We have shown that $C\in\cC_i$.

It is clear that $f$ is an injective mapping. To complete the proof of the 
claim, it only remains to show that $f$ is surjective, that is, for every 
$C\in\cC_i$ there are chains $(C_1,C_2)\in\cCa\times\cCb$ with $f(C_1,C_2)=C$.

Choose a chain $C\in\cC_i$. Since $C$ is split, it must contain the element 
$H+t$. Call the elements of $C$ contained in $H+t$ the \emph{bottom part} of 
$C$, and the elements containing $H+t$ the \emph{top part} of~$C$. The 
definition of split chain further implies that all the elements in the top 
part have the same head~$H$ and all the elements in the bottom part have the 
same tail~$t$. Hence, the bottom part of the chain $C$ has the form $C_1+t$ 
for some chain $C_1\in\chain{h}{H}$. Similarly, the top part has the form 
$H+C_2$ for a chain $C_2\in\chain{t}{T}$. Since $C$ is $\one$-tight, we may 
use Lemmas~\ref{obs-tight} and~\ref{obs-tight2} to see that $C_1$ is 
$\one$-tight and $C_2$ is $H$-tight, showing that 
$(C_1,C_2)\in\cCa\times\cCb$. Since $f(C_1,C_2)=C$, we see that $f$ is the 
required bijection.
\end{proof}       

We now have all the necessary ingredients to finish the proof of 
Proposition~\ref{pro-second}. Let us write $H=\pi_1$ and $T=\psu 1$. From our results,
we get
\begin{align*}
\mu(\sigma,\pi)
&=w(\chain{\sigma}{\pi})&\\
&=w(\cC^{*})&&\text{by Lemma~\ref{lem-split}}\\
&=\sum_{i=1}^m w(\cC_i)&&\text{by Claim~\ref{cla-c0}}\\
&=\sum_{i=1}^m w(\tch{\spr i}{H}{\one}) w(\tch{\ssu i}{T}{H})&&\text{by Claim~\ref{cla-concat}}\\                                   
&=\sum_{i=1}^m \mu(\spr i, H)w(\tch{\ssu i}{T}{H})&&\text{by first part of
Lemma~\ref{lem-tight}}\\
&=\sum_{i=1}^m \mu(\spr i, H)\sum_{j=0}^{k-1}\mu(\ssu i, \suf T j)&&\text{by 
Lemma~\ref{lem-tight2}}\\                                      
&=\sum_{i=1}^m \sum_{j=1}^{k}\mu(\spr i, \pi_1)\mu(\ssu i, \psu 
j)&&\text{since $\suf T j= \psu{j+1}$}
\end{align*}                          
Thus, Proposition~\ref{pro-second} is now proved.

We now present some consequences of Propositions~\ref{pro-first}
and~\ref{pro-second}.

\begin{corollary}\label{cor-algo}
There is an algorithm that, given two separable permutations $\sigma$ and $\pi$,
computes the value of $\mu(\sigma,\pi)$ in time polynomial in~$|\sigma|+|\pi|$.
\end{corollary}
\begin{proof}
Let $\pi=\pi_1\pi_2\dotsb\pi_n$ be a separable permutation. For two integers
$i,j$ with $1\le i\le j\le n$, let $\pi[i,j]$ denote the subpermutation of
$\pi$ order-isomorphic to the sequence $\pi_i,\pi_{i+1},\dotsc,\pi_j$. Note
that $\pi[i,j]$ is also separable. We call $\pi[i,j]$ a \emph{range
subpermutation} of~$\pi$.

Suppose that $\sigma=\sigma_1\dotsb\sigma_m$ and $\pi=\pi_1\dotsb\pi_n$ are two
separable permutations. Our goal is to compute $\mu(\sigma,\pi)$. We will use 
a straightforward dynamic programming algorithm to perform this computation.
We will compute all the values of the form $\mu(\sigma[i,j], \pi[k,\ell])$, for all quadruples $(i,j,k,\ell)$
satisfying $1\le i\le j\le m$ and $1\le k\le \ell\le n$.  For each such
quadruple $(i,j,k,\ell)$ we store the value of $\mu(\sigma[i,j],\pi[k,\ell])$
once we compute it, so that we do not need to compute this value more than
once, even though we may need it several times to compute other values of~$\mu$.

There are $\cO(m^2n^2)$ quadruples $(i,j,k,\ell)$ to consider, and for
each such quadruple, we may use Propositions~\ref{pro-first}
and~\ref{pro-second} to express $\mu(\sigma[i,j], \pi[k,\ell])$ as a
combination of polynomially many values of the form
$\mu(\sigma[i',j'],\pi[k',\ell'])$ where $\sigma[i',j']$ and
$\pi[k',\ell']$ are range subpermutations of $\sigma[i,j]$ and
$\pi[k,\ell]$ with $\pi[k',\ell']\neq\pi[k,\ell]$.  Therefore, we can
in polynomial time compute all the values of the form
$\mu(\sigma[i,j],\pi[k,\ell])$, including
$\mu(\sigma,\pi)=\mu(\sigma[1,m],\pi[1,n])$.
\end{proof}

Note that the number of permutations belonging to an interval $[\sigma,\pi]$ may
in general be exponential in the size of $\pi$, even when $\pi$ and $\sigma$
are separable. Therefore, computing the M\"obius function $\mu(\sigma,\pi)$
directly from equation~\eqref{eq:eq-mob} would be much less efficient than the
algorithm of the previous corollary.

Let us say that a class of permutations $\clc$ is \emph{sum-closed} if for each
$\pi,\sigma\in\clc$, the class $\clc$ also contains $\pi+\sigma$. Similarly,
$\clc$ is \emph{skew-closed} if $\pi,\sigma\in\clc$ implies $\pi*\sigma\in\clc$.
For a set $\clp$ of permutations, the \emph{closure} of $\clp$, denoted by
$\cl{\clp}$, is the smallest sum-closed and skew-closed class of permutations
that contains~$\clp$. Notice that $\cl{\{\one\}}$ is exactly the set of
separable permutations.

The next corollary is an immediate consequence of
Propositions~\ref{pro-first} and~\ref{pro-second} (see also Corollary
\ref{coro-first-recurrence}), and we omit its proof.

\begin{corollary}\label{cor-bounded} Suppose that $\sigma$ is a permutation that
is neither decomposable nor skew-decomposable. Let $\clp$ be any set of
permutations. Then 
\[
\max\{|\mu(\sigma,\pi)|;\; \pi\in \clp\}=\max\{|\mu(\sigma,\pi)|;\; \pi\in
\cl\clp\}.
\]
Moreover, the computation of $\mu(\sigma,\pi)$ for $\pi\in\cl\clp$ can be
efficiently reduced to the computation of the values $\mu(\sigma,\rho)$ for
$\rho\in\clp$.
\end{corollary}

\section{The M\"obius function of separable permutations}\label{sec-sep}

Let us now consider the values of $\mu(\sigma,\pi)$ for separable
permutations $\sigma$ and~$\pi$. Our goal is to show that the values
of the M\"obius function in the poset of separable permutations have a
combinatorial interpretation in terms of the so-called \emph{normal
  embeddings}, which we define below. This alternative interpretation
of the M\"obius function generalizes previous results of Sagan and
Vatter~\cite{sagan-vatter} for the M\"obius function of intervals of
layered permutations, which we explain at the end of this section.

As a consequence of this new interpretation of the M\"obius function,
we are able to relate the M\"obius function $\mu(\sigma,\pi)$ to the
number of occurrences of $\sigma$ in $\pi$, by showing that
$|\mu(\sigma,\pi)|\le\sigma(\pi)$.  We also show that $\mu(\one,\pi)$
is equal to $-1$, $0$ or $1$ whenever $\pi$ is separable.

The recursive structure of separable permutations makes it convenient
to represent a separable permutation by a tree that describes how the
permutation may be obtained from smaller permutations by sums and skew
sums. We now formalize this concept. A \emph{separating tree} $T$ is a
rooted tree $T$ with the following properties:
\begin{itemize}
\item Each internal node of $T$ has one of two types: it is either a
\emph{direct node} or a \emph{skew node}.
\item Each internal node has at least two children. The children of a given
internal node are ordered into a sequence from left to right.
\end{itemize}

Each separating tree $T$ represents a unique separable permutation $\pi$,
defined recursively as follows:
\begin{itemize}
\item If $T$ has a single node, it represents the singleton
  permutation~$\one$.
\item Assume $T$ has more than one node. Let $N_1,\dotsc, N_k$ be the
  children of the root in their left-to-right order, and let $T_i$
  denote the subtree of $T$ rooted at the node $N_i$. Let
  $p_1,\dotsc,p_k$ be the permutations represented by the trees
  $T_1,\dotsc,T_k$.  Then $T$ represents the permutation
  $p_1+\dotsb+p_k$ if the root of~$T$ is a direct node and
  $p_1*\dotsb*p_k$ if the root of $T$ is a skew node.
\end{itemize}

Note that the leaves of $T$ correspond bijectively to the letters of~$\pi$. In
fact, when we perform a depth-first left-to-right traversal of $T$, we encounter
the leaves in the order that corresponds to the left-to-right order of the
letters of~$\pi$. See Figure~\ref{fig-qp} for an example.

A given separable permutation may be represented by more than one
separating tree. A separating tree is called a \emph{reduced tree} if
it has the property that the children of a direct node are leaves or
skew nodes, and the children of a skew node are leaves or direct
nodes.  Each separable permutation $\pi$ is represented by a unique
reduced tree, denoted by~$T(\pi)$. We assume that each leaf of $T$ is
labelled by the corresponding letter of~$\pi$.

This slightly modified concept of separating tree and its relationship
with separable permutations have been previously studied in
algorithmic contexts~\cite{bbl,ys}. We will now show that the reduced
tree allows us to obtain a simple formula for the M\"obius function of
separable permutations.

Let $[n]$ denote the set $\{1,\dotsc,n\}$.  Let
$\pi=\pi_1\pi_2\ldots\pi_n$ and
$\sigma=\sigma_1\sigma_2\ldots\sigma_m$ be two permutations, with
$\sigma\le \pi$. An \emph{embedding} of $\sigma$ into $\pi$ is a
function $f\colon [m]\to[n]$ with the following two properties:
\begin{itemize}
\item for every $i,j\in [m]$, if $i<j$ then $f(i)<f(j)$ (so $f$ is
  monotone increasing).
\item for every $i,j\in [m]$, if $\sigma_i<\sigma_j$, then
  $\pi_{f(i)}<\pi_{f(j)}$ (so $f$ is order-preserving).
\end{itemize}

Let $f$ be an embedding of $\sigma$ into~$\pi$. We say that a leaf
$\ell$ of $T(\pi)$ is \emph{covered} by the embedding $f$ if the
letter of $\pi$ corresponding to $\ell$ is in the image of~$f$. A leaf
is \emph{omitted} by $f$ if it is not covered by~$f$. An
\emph{internal node} is a node that is not a leaf.  An internal node
$N$ of $T(\pi)$ is \emph{omitted} by $f$ if all the leaves in the
subtree rooted at $N$ are omitted. A node is \emph{maximal omitted},
if it is omitted but its parent in $T(\pi)$ is not omitted.

Assume that $\pi$ is a separable permutation and $T(\pi)$ its reduced
tree.  Two nodes $N_1$ and $N_2$ of a tree $T(\pi)$ are called
\emph{twins} if they are siblings having a common parent $P$, they
appear consecutively in the sequence of children of $P$, and the two
subtrees of $T$ rooted at $N_1$ and $N_2$ are isomorphic, that is,
they only differ by the labeling of their leaves, but otherwise have
the same structure. In particular, any two adjacent leaves are twins.

A \emph{run} under a node $N$ in $T$ is a maximal sequence
$N_1,\dotsc, N_k$ of children of $N$ such that each two consecutive
elements of the sequence are twins. Note that the sequence of children
of each internal node is uniquely partitioned into runs, each possibly
consisting of a single node.  A \emph{leaf run} is a run whose nodes
are leaves, and a \emph{non-leaf run} is a run whose nodes are
non-leaves.  The first (leftmost) element of each run is called
\emph{the leader} of the run and the remaining elements are called
\emph{followers}.

Using the tree structure of $T(\pi)$, we will show that
$\mu(\sigma,\pi)$ can be expressed as a signed sum over a set of
embeddings of $\sigma$ into~$\pi$ that have a special
structure. Following the terminology of Sagan and
Vatter~\cite{sagan-vatter}, we call these special embeddings
\emph{normal}.

\begin{definition} Let $\sigma$ and $\pi$ be separable permutations, let
$T(\pi)$ be the reduced tree of~$\pi$. An embedding $f$ of $\sigma$ into $\pi$
is called \emph{normal} if it satisfies the following two conditions.
\begin{itemize}
\item If a leaf $\ell$ is maximal omitted by $f$, then $\ell$ is the leader of
its corresponding leaf run. 
\item If an internal node $N$ is maximal omitted by $f$, then $N$ is a follower
in its non-leaf run.
\end{itemize}
\end{definition} 

Let $N(\sigma,\pi)$ denote the set of normal embeddings of $\sigma$ into~$\pi$.
The \emph{defect} of an embedding $f\in N(\sigma,\pi)$, denoted by $d(f)$, is
the number of leaves that are maximal omitted by~$f$. The \emph{sign} of $f$,
denoted by $\sgn(f)$, is defined as $(-1)^{d(f)}$.

We now present our main result.

\begin{theorem}\label{thm-main} If $\sigma$ and $\pi$ are (possibly empty)
separable permutations, then
\[
\mu(\sigma,\pi)=\sum_{f\in N(\sigma,\pi)} \sgn(f).
\] 
\end{theorem}

Consider, as an example, the two permutations $\pi$ and $\sigma$ depicted on
Figure~\ref{fig-qp}. The children of the root of $T(\pi)$ are partitioned into
three runs, where the first run has three internal nodes, the second run has a
single leaf, and the last run has a single internal node. Accordingly, there are
five normal embeddings of $\sigma$ into $\pi$, depicted in
Figure~\ref{fig-good}. Of these five normal embeddings, two have sign -1 and
three have sign 1, giving $\mu(\sigma,\pi)=1$.

\begin{figure}[htbp]
\hfil \includegraphics{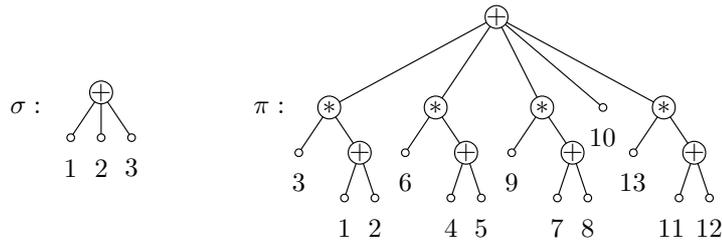}
 \caption{The separating trees of two permutations $\sigma$ and 
 $\pi$}\label{fig-qp}
\end{figure}

\begin{figure}[htbp]
\hfil\includegraphics{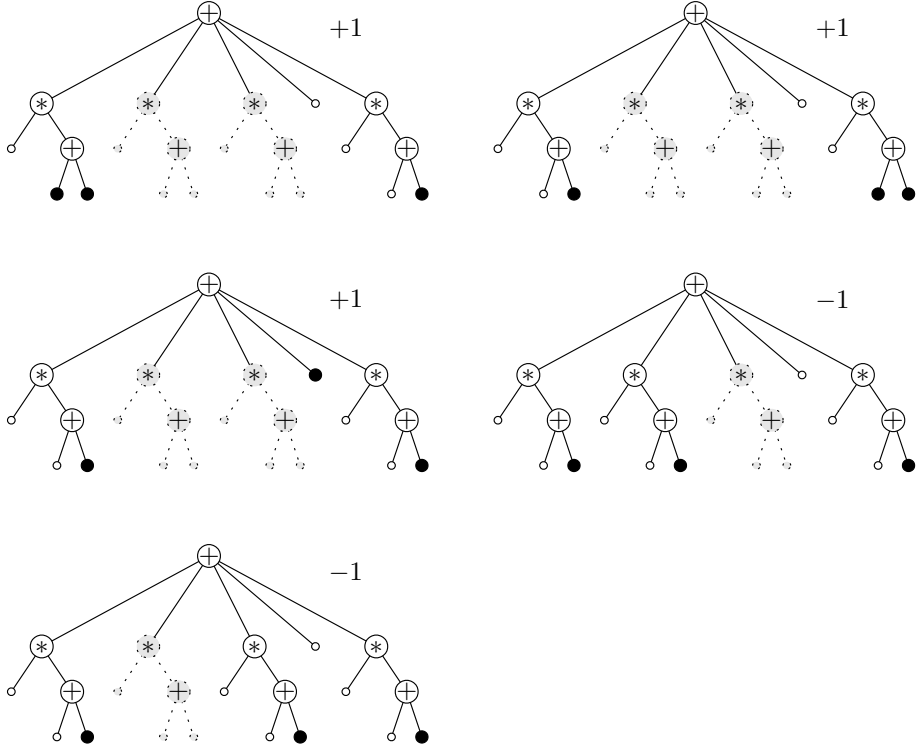}
\caption{The normal embeddings of $\sigma$ in $\pi$ (see
  Figure~\ref{fig-qp}), together with their signs.  The leaves covered
  by the embedding are represented by black disks, the leaves that are
  maximal omitted are represented by empty circles. Dotted lines
  represent subtrees rooted at a maximal omitted internal node. Note
  that the leaves of such subtrees do not contribute to the sign of
  the embedding.}\label{fig-good}
\end{figure}

\begin{proof}[Proof of Theorem~\ref{thm-main}]
Let $\bmu(\sigma,\pi)$ denote the value of $\sum_{f\in N(\sigma,\pi)}
\sgn(f)$. Our goal is to prove that $\bmu(\sigma,\pi)$ is equal to
$\mu(\sigma,\pi)$. We proceed by induction on $|\pi|$. For $\sigma=\pi$, we
clearly have $\bmu(\sigma,\pi)=\mu(\sigma,\pi)=1$, and if $\pi$ does not contain
$\sigma$, then $\bmu(\sigma,\pi)=\mu(\sigma,\pi)=0$. 

Suppose now that $\sigma<\pi$. Since $\pi$ is separable, it is
decomposable or skew-decomposable.  Assume, without loss of
generality, that $\pi$ is decomposable. Let $\pi_1+\dotsb+\pi_n$ be
its decomposition. Since the values of $\mu(\sigma,\pi)$ are uniquely
determined by the recurrences of Proposition~\ref{pro-first}
and~\ref{pro-second}, it is enough to show that $\bmu$ satisfies the
same recurrences.

Consider first the case when $\pi_1=\one$, which is treated by
Proposition~\ref{pro-first}. Let $\sigma_1+\dotsb+\sigma_m$ be the
decomposition of $\sigma$, let $k=\deg(\pi)$ and let $\ell=\deg(\sigma)$. This
means that the leftmost $k$ leaves of $T(\pi)$ are all children of the root
node, and they form a leaf run. Therefore, in any normal embedding, all the
$k-1$ leaves representing $\pi_2,\dotsc,\pi_k$ are covered, because they are
followers of~$\pi_1$. Necessarily, any element of $\sigma$ that is embedded to
one of the first $k$ elements of $\pi$ must be one of the first $\ell$
elements of~$\sigma$. Consequently, if $k-1>\ell$, there is no normal embedding
of $\sigma$ into $\pi$, and $\bmu(\sigma,\pi)=0$.

Suppose now that $k-1=\ell$. Then, in any normal embedding $f\in
N(\sigma,\pi)$, the element $\pi_1$ is omitted, the elements representing
$\sigma_1,\dotsc,\sigma_{k-1}$ are embedded on $\pi_2,\dotsc,\pi_k$, and the
elements of $\ssu{k-1}$ are embedded to the elements $\psu k$. The restriction
of $f$ to $\ssu{k-1}$ is a normal embedding $f'$ from the set $N(\ssu{k-1},\psu
k)$, and conversely, a normal embedding $f'$ from $N(\ssu{k-1},\psu k)$ can be
uniquely extended into an embedding $f\in N(\sigma,\pi)$. We then have
$d(f)=1+d(f')$, because $\pi_1$ is the only maximal omitted leaf of $f$ that is
not a maximal omitted leaf
of~$f'$. This shows that $\bmu(\sigma, \pi)=-\bmu(\ssu{k-1},\psu k)$.

Assume now that $k-1<\ell$. Let $N^+(\sigma,\pi)$ denote the set of normal
embeddings of $\sigma$ into $\pi$ that cover the element $\pi_1$, and let
$N^-(\sigma,\pi)$ be the set of those that omit~$\pi_1$. By the same argument
as in the previous paragraph, we see that $N^+(\sigma,\pi)$ is mapped by a
sign-preserving bijection to $N(\ssu k, \psu k)$, and $N^-(\sigma,\pi)$ is
mapped by a sign-reversing bijection to $N(\ssu{k-1},\psu k)$. Consequently,
$\bmu(\sigma,\pi)=\bmu(\ssu k, \psu k)-\bmu(\ssu{k-1},\psu k)$. 

These arguments show that $\bmu$ satisfies the recurrences of
Proposition~\ref{pro-first}.

Assume now that $\pi_1>\one$, which corresponds to the situation of
Proposition~\ref{pro-second}. Let $\pi_1+\dotsb+\pi_n$ be the
decomposition of $\pi$, let $\sigma_1+\dotsb+\sigma_m$ be the decomposition of
$\sigma$, and let $k\in[n]$ be the largest integer such that
$\pi_1=\dotsb=\pi_k$.
The $n$ blocks of $\pi$ correspond precisely to $n$ children of the root of the
tree $T(\pi)$, and the leftmost $k$ blocks form a non-leaf run. Therefore, each
normal embedding $f\in N(\sigma,\pi)$ must cover the leftmost child of the
root, which represents $\pi_1$, but it may omit some of its followers, which
represent the blocks $\pi_2,\dotsc,\pi_k$. Note that the symbols of $\sigma$
that are embedded into $\pi_1$ by $f$ must form a prefix of the form $\spr i$,
for some $i\in[m]$. 

For $f\in N(\sigma,\pi)$, let $I(f)\in[m]$ be the largest number $i$ such that
all the symbols of $\spr i$ are embedded into $\pi_1$, and let $J(f)\in[k]$ be
the largest number $j$ such that among the leftmost $j$ children of the
root of $T(\pi)$, only the node representing $\pi_1$ is covered. Let $N_{i,j}$
be the set $\{f\in N(\sigma, \pi)\colon I(f)=i,\, J(f)=j\}$. Notice that an
embedding $f\in N_{i,j}$ decomposes in an obvious way into a normal embedding
$f_1\in N(\spr i,\pi_1)$ and a normal embedding $f_2\in N(\ssu i, \psu j)$, and
that we have $d(f)=d(f_1)+d(f_2)$, and hence $\sgn(f)=\sgn(f_1)\sgn(f_2)$. This
decomposition is a bijection between $N_{i,j}$ and $N(\spr i,\pi_1)\times
N(\ssu i, \psu j)$. Consequently, we have the identity
\[
\sum_{f\in N_{i,j}} \sgn(f)=\sum_{f_1\in N(\spr i,\pi_1)}\sum_{f_2\in N(\ssu i,
\psu j)} \sgn(f_1)\sgn(f_2)=\bmu(\spr i, \pi_1)\bmu(\ssu i, \psu j).
\]
Summing this identity for each $i\in[m]$ and each $j\in[k]$, we conclude that 
\[
 \bmu(\sigma,\pi)=\sum_{i=1}^m\sum_{j=1}^k \bmu(\spr i, \pi_1)\bmu(\ssu i, \psu
j),
\]
which is the recurrence of Proposition~\ref{pro-second}. Therefore,
$\bmu(\sigma, \pi)=\mu(\sigma,\pi)$.
\end{proof}

Let us now state several consequences of Theorem~\ref{thm-main}.

\begin{corollary} If $\pi$ is separable, then $\mu(\one,\pi)\in\{0,1,-1\}$.  
\end{corollary}
\begin{proof} The permutation $\one$ can have at most one normal
  embedding into~$\pi$.  Namely, if $|\pi|>1$, then $T(\pi)$ has at
  least one leaf $\ell$ that is not a leader of its leaf run, but each
  of its ancestors is a leader of its non-leaf run. Such a leaf $\ell$
  must be covered by any normal embedding of any permutation
  into~$\pi$.
\end{proof}

The next corollary confirms a (more general version of a) conjecture of
Steingr\'\i msson and Tenner~\cite{ste-tenner}.
\begin{corollary} If $\pi$ and $\sigma$ are separable permutations, then
$|\mu(\sigma,\pi)|$ is at most the number of occurrences of $\sigma$ in~$\pi$.
\end{corollary}
\begin{proof} This follows from the fact that the number of occurrences of
$\sigma$ 
in $\pi$ is clearly at least the number of normal embeddings of $\sigma$
into~$\pi$.
\end{proof}

Using Theorem \ref{thm-main}, it is easy to show that for
$\pi_n=214365\cdots (2n)(2n-1)$, we have $\mu(12,\pi_n)=n-1$.  Thus,
the following result.

\begin{corollary} 
  The value of the M\"obius function on intervals $[\sigma,\pi]$ is
  unbounded, even for separable permutations $\sigma$ and $\pi$.
\end{corollary}

Recall that a permutation is \emph{layered} if it is the concatenation
of decreasing sequences, such that the letters in each sequence are
smaller than all letters in subsequent sequences. One example is the
permutation 21365487, whose layers are shown by 21--3--654--87. Sagan
and Vatter \cite{sagan-vatter} gave a formula for the M\"obius
function of intervals of layered permutations, and it is easy to see
that layered permutations are special cases of separable permutations.
Namely, a layered permutation is separable, and its separating tree
has depth 2 (except in the trivial cases of the increasing and
decreasing permutations), where the children of the root are the
layers of the permutation, and the grandchildren of the root are all
leaves.

\section{Concluding remarks, conjectures and open problems}\label{sec-last}

We have shown in Corollary~\ref{cor-algo} that $\mu(\sigma,\pi)$ can
be computed efficiently when $\sigma$ and $\pi$ are separable. The
same argument does in fact apply in a more general form: For any
hereditary class $\clc$ of permutations that is a closure of a finite
set of permutations, there is a polynomial-time algorithm to compute
$\mu(\sigma,\pi)$ for a given $\sigma$ and $\pi$ in~$\clc$. We do not
know whether such an algorithm also exists for more general classes of
permutations.

Bose, Buss and Lubiw~\cite{bbl} have shown that it is NP-hard for
given permutations $\pi$ and $\sigma$ to decide whether $\pi$ contains
$\sigma$. In view of this, it seems unlikely that $\mu(\sigma,\pi)$ could be
computed efficiently for general permutations $\sigma$ and~$\pi$.

Our results imply that for a separable permutation $\pi$, the M\"obius
function $\mu(\one,\pi)$ has absolute value at most 1. In fact, the class
of separable permutations is the unique largest hereditary class with this
property, since any hereditary class not contained in the class of
separable permutations must contain 2413 or 3142, and
$\mu(\one,2413)=\mu(\one,3142)=-3$. It is natural to consider $\mu(\one,\pi)$
as a function of $\pi$, and ask whether this function is bounded on a given
class of permutations. By Corollary~\ref{cor-bounded}, if a hereditary class $\clc$ is a
closure of a finite set of permutations, then $\mu(\one,\pi)$ is bounded on~$\clc$.
We do not know if there is another example of a permutation class on which this
function is bounded. 

On the other hand, we do not have a proof that $\mu(\one,\pi)$ is
unbounded on the set of all permutations, although numerical evidence
suggests that this is the case. According to our computations, the
sequence of maximum values of $|\mu(\one,\pi)|$ for $\pi\in\clsn$,
starting at $n=1$, begins $1, -1, 1, -3, 6, -11, 15, -27, -50, -58,
143, \ldots$.  For these cases ($n\le11)$, there is, up to trivial
symmetries, a unique permutation for which the M\"obius function
attains this value.  These permutations are
\begin{eqnarray*}
&&  1,\;\;12,\;\;132,\;\;2413,\;\;24153,\;\;351624,\;\;2461735,\;\;35172846,\\ &&
  472951836,\;\; 4\,6\,8\,1\,9\,2\,10\,3\,5\,7,\;\;
  3\,6\,1\,9\,4\,11\,7\,2\,10\,5\,8
\end{eqnarray*}
All of the above permutations are \emph{simple} (except for 132, but
there are no simple permutations of length 3).  A permutation is
simple if it has no segment $a_ia_{i+1}\ldots a_{i+k}$ where $1\le
k<n-1$ and $\{a_i,a_{i+1},\ldots,a_{i+k}\}$ is a set of consecutive
integers (see \cite{albert-atkinson}).  Thus, in some (imprecise)
sense, simple permutations are the opposite of (skew) decomposable
permutations (and, in particular, separable permutations).  In
particular, a simple permutation can neither be decomposed nor skew
decomposed.  We are not able to compute $\mu(1,\pi)$ for all
permutations $\pi$ of length 12, but for simple permutations $\pi$ the
maximum value of $\mu(1,\pi)$ is $-261$, for
$\pi=4\;7\;2\;10\;5\;1\;12\;8\;3\;11\;6\;9$.

In light of Corollary \ref{coro-not-high}, to show that
$|\mu(\one,\pi)|$ is unbounded, it would suffice to show that the
maximum of $|\mu(\one,\pi)|$ for permutations $\pi$ of length $n$, for
any $n$, is attained only by a permutation $\pi$ that does not start
with 1.  In that case $|\mu(\one,\one+\pi)| = |\mu(\one,\pi)|$, so
there would be a permutation $\tau$ of length $n+1$ for which
$|\mu(\one,\tau)|> |\mu(\one,\one+\pi)| = |\mu(\one,\pi)|$.

\begin{question}
  For which permutation classes $\clc$ is the function $\mu(\one,\pi)$
  bounded on $\clc$? Is $\mu(\one,\pi)$ unbounded on the set of all
  permutations? Can non-trivial upper or lower bounds be found for
  $max_{\pi\in\clsn}|\mu(\one,\pi)|$, as a function of~$n$?
\end{question}

We have exhibited several classes of intervals whose M\"obius function
is zero (and more were presented in \cite{ste-tenner}).  Can the
following question be answered precisely?

\begin{question}
  When is $\mu(\sigma,\pi)=0$?
\end{question}

For separable permutation $\pi$, we have shown that $|\mu(\sigma,\pi)|$ is
at most $\sigma(\pi)$, that is, the number of occurrences of $\sigma$ in~$\pi$.
This is not true for non-separable $\pi$, even when $\sigma=\one$, as shown
above. However, it might be possible to bound $|\mu(\sigma,\pi)|$ as a function
of $\sigma(\pi)$.
\begin{question}
Is there an upper bound for $|\mu(\sigma,\pi)|$ that only depends
on~$\sigma(\pi)$?
\end{question}

The following conjecture has been verified for $n\le10$.

\begin{conjecture}
The maximum value of the M\"obius function $\mu(\sigma,\pi)$ for
separable permutations $\sigma$ and $\pi$, where $\pi$ has length $n\ge3$,
is given by 
$$
\max_k\binom{n-1-k}{k}.
$$ This maximum is attained by the permutation $\pi$ that starts with
its odd letters in increasing order, followed by the even letters in
decreasing order, and $\sigma$ of the same form and length
$2\cdot\lfloor(n+1/2)\rfloor$ if the length of $\pi$ is $2n$, and
$2\cdot\lfloor(n+1/2)\rfloor-1$ if the length of $\pi$ is $2n-1$.
\end{conjecture}
As an example, $\mu(13542,135798642)=15=\binom{9-1-2}{2}$.

\def\dsp{\ensuremath{\Delta(\sigma,\pi)}}

Finally, we mention some questions about the topology of the order
complexes of intervals in the poset $\clp$.  (For definitions, see
\cite{ECI}). Given an interval $[\sigma,\pi]$, let $\dsp$ be the order
complex of the poset obtained from $[\sigma,\pi]$ by removing $\sigma$
and~$\pi$.  

\begin{question}\label{q-topo}
  \begin{enumerate}
  \item For which $\sigma$ and $\pi$ does $\dsp$ have the homotopy
    type of a wedge of spheres?
  \item\label{q-2} Let $\Gamma$ be the subcomplex of $\dsp$ induced by those
    elements $\tau$ of $[\sigma,\pi]$ for which $\mu(\sigma,\tau)=0$.
    Is $\Gamma$ a pure complex?  
  \item \label{q-3} If $\sigma$ occurs precisely once in $\pi$, and
    $\mu(\sigma,\pi)=\pm1$, is $\dsp$ homotopy equivalent to a sphere?
\item For which $\sigma$ and $\pi$ is $\dsp$ shellable?
  \end{enumerate}
\end{question}

We should mention that for $\sigma=231$ and $\pi=231564$, the order
complex $\dsp$ is not shellable; it consists of two connected
components, each of which is contractible.  However, removing from
$[231,231564]$ all elements $\tau$ with $\mu(231,\tau)=0$, we obtain a
shellable complex, namely a four-element boolean algebra.  For parts
(\ref{q-2}) and (\ref{q-3}) in Question \ref{q-topo}, we know no
counterexamples.  Since we have so far only examined intervals of low
rank, our evidence is not strong.

\end{document}